\documentclass[12pt,a4paper]{amsart}
\usepackage{fancyhdr}
\usepackage{appendix}
\usepackage{amssymb,amscd,amsxtra,calc}
\usepackage{mathrsfs}
\usepackage{amsmath}
\usepackage{multirow}
\usepackage{verbatim}
\usepackage{mathtools}
\usepackage[all]{xy}
\usepackage[colorlinks,linkcolor=red,anchorcolor=blue,citecolor=blue]{hyperref}
\usepackage{tikz}
\usepackage{tikz-cd}
\theoremstyle{plain}
    \newtheorem{thm}{Theorem}[section]
    \renewcommand{\thethm}
    {\arabic{section}.\arabic{thm}}
    
     \newtheorem{conjecture}[thm]{Conjecture}
    
    \newtheorem{lemma}[thm]{Lemma}
    \newtheorem{proposition}[thm]{Proposition}
    \newtheorem{question}[thm]{Question}
    \newtheorem{theorem}[thm]{Theorem}

\theoremstyle{definition}

    \newtheorem{notation}[thm]{Notation}
        \newtheorem{example}[thm]{Example}
    \newtheorem*{notation*}{Notation and Terminology}
    \newtheorem{remark}[thm]{Remark}

\theoremstyle{remark}

\DeclareMathOperator{\Sym}{Sym}

\makeatletter

\newcommand{\Rmnum}[1]{\expandafter\@slowromancap\romannumeral #1@}
\makeatother

\begin{document}
\bibliographystyle{alpha}
\title[Dynamical rigidity]
{Bigness of tangent bundles and dynamical rigidity of Fano manifolds of Picard number 1 \\[1em] \normalfont\footnotesize With an appendix by Jie Liu}

\author{Feng Shao} 
\author{Guolei Zhong}

\address{Jie Liu, Institute of Mathematics, Academy of Mathematics and Systems Science, Chinese Academy of Sciences, Beijing, 100190, China}
\email{jliu@amss.ac.cn}

\address
{
\textsc{Feng Shao, Center for Complex Geometry, Institute for Basic Science (IBS), 55 Expo-ro,  Yuseong-gu, Daejeon, 34126, Republic of Korea.
}}
\email{shaofeng@amss.ac.cn, shaofeng@ibs.re.kr}
\address
{
\textsc{Guolei Zhong, Center for Complex Geometry, Institute for Basic Science (IBS), 55 Expo-ro,  Yuseong-gu, Daejeon, 34126, Republic of Korea.
}}
\email{zhongguolei@u.nus.edu, guolei@ibs.re.kr}

\begin{abstract}
Let \(f\colon X\to Y\) be a surjective morphism of  Fano manifolds of Picard number 1 whose VMRTs at a general point are not dual defective.
Suppose that the tangent bundle \(T_X\) is big. 
We show that \(f\) is an isomorphism unless \(Y\) is a projective space. 
As applications, we explore the bigness of the tangent bundles of complete intersections, del Pezzo manifolds, and Mukai manifolds, as well as their dynamical rigidity.
\end{abstract}
\subjclass[2010]{
14J40, 14J45.
}

\keywords{endomorphism, big tangent bundle, Fano manifold, VMRT}

\maketitle

\tableofcontents

\section{Introduction}
We work over the field $\mathbb{C}$ of complex numbers. 
A classical question in algebraic and complex geometry asks for a description of a Fano manifold of Picard number 1 admitting a non-isomorphic surjective endomorphism. 
The following folklore conjecture of the 1980s predicts that the projective space is the only possibility for the existence of such a non-trivial endomorphism. 
This conjecture serves as one of the starting points of this paper. 

\begin{conjecture}\label{main-conj-pn}
Let $X$ be a Fano manifold of Picard number 1. 
Suppose that $X$ admits a non-isomorphic surjective endomorphism.
Then $X$ is a projective space.
\end{conjecture}

Conjecture \ref{main-conj-pn} has been verified in low dimension and  many other special cases:
\begin{enumerate}
\item Almost homogeneous spaces (\cite{PS89}, \cite{HN11});
\item Smooth  hypersurfaces of a projective space (\cite{PS89}, \cite{Bea01}; cf.~\cite{KT23});
\item Fano threefolds  (\cite{ARV99}, \cite{HM03}; cf.~\cite{KT23});
\item Fano manifolds containing a rational curve with trivial normal bundle (\cite{HM03});
\item Fano fourfolds with Fano index \(\geq 2\) (\cite{SZ22}; cf.~\cite{KT23}); and 
\item Del Pezzo manifolds, i.e., the Fano index \(=\dim(X)-1\) (\cite{SZ22}, \cite{AKP08}).
\end{enumerate}

In this paper, we investigate Conjecture \ref{main-conj-pn} for the case when the tangent bundle \(T_X\) is big.
Recall that the tangent bundle \(T_X\) of a smooth projective variety \(X\) is called \textit{big} (resp. \textit{pseudo-effective}) if the tautological line bundle \(\mathcal{O}_{\mathbb{P}(T_X)}(1)\) of the projectivized tangent bundle (in the sense of Grothendieck) is big (resp. pseudo-effective). 
After Mori's magnificent solution \cite{Mor79} to the Hartshorne conjecture, it has become evident that a certain positivity  condition of the tangent bundle would impose 
restrictive geometry on the underlying variety. 
In the spirit of this expectation, our first main result below gives a dynamical rigidity when the tangent bundle of the source is big and the variety of minimal rational tangents (VMRT for short) is not dual defective along a general point of $X$ and of $Y$.

\begin{theorem}\label{thm-main}
Let \(X\) and \(Y\) be the Fano manifolds of Picard number 1. 
Suppose that the VMRT \(\mathcal{C}_x\subseteq\mathbb{P}(\Omega_{X,x})\) (resp. \(\mathcal{C}_y\subseteq\mathbb{P}(\Omega_{Y,y})\)) at a general point \(x\in X\) (resp. \(y\in Y\)) is not dual defective. 
Suppose further that the tangent bundle \(T_X\) is big.
Then any surjective morphism \(X\to Y\) has to be an isomorphism unless \(Y\) is a projective space; in particular, \(X\) admits no non-isomorphic surjective endomorphism unless it is a projective space.
\end{theorem}
Let \(Z\subseteq\mathbb{P}^N=\mathbb{P}(V)\) be a projective variety. The closure of the set of all tangent hyperplanes of \(Z\) is called the \textit{dual variety} \(\check{Z}\subseteq\mathbb{P}^N=\mathbb{P}(V^*)\), where \(V^*\) is the dual vector space of \(V\).
We say that \(Z\) is \textit{dual defective}, if \(N-1-\dim(\check{Z})>0\) holds (cf.~\cite{Tev05}). 
Indeed, a projective variety is dual defective only in very special cases and our non-dual defective assumption in Theorem  \ref{thm-main} is not very restrictive, for example, it is known that any Fano threefold of Picard number 1 has a dominating family of minimal rational curves whose VMRT along a general point is not dual defective (see \cite{HM03}).
Various examples satisfying the assumption in Theorem \ref{thm-main} can be found; see \cite[Introduction]{Liu23} for a nice summary of recent progress (cf.~Section \ref{s:ex}).

In \cite{HM99}, Hwang and Mok gave an affirmative answer to a question raised by Lazarsfeld (see \cite{Laz84}), asserting that any surjective morphism from a rational homogeneous space of Picard number 1 to a smooth projective variety which is different from the projective space, has to be an isomorphism. 
Inspired by their work and the fact that every rational homogeneous space has a big tangent bundle (cf. Example \ref{exa-rational-homogeneous}), as an extension of Lazarsfeld's question, it is natural to ask the following question.

\begin{question}\label{ques-laz}
Let $f\colon X\to Y$ be a finite surjective morphism between Fano manifolds of Picard number 1. Suppose that  $T_X$ is big and $Y$ is not isomorphic to a projective space. Is $f$ an isomorphism?
\end{question}

Note that, in addition to rational homogeneous spaces,  Question \ref{ques-laz} also has a positive answer if $\dim X\leq 3$ by \cite[Corollary 1.2]{HL23}, \cite{Laz84} and Theorem \ref{thm-main}.

To establish Theorem \ref{thm-main}, we give the following  Theorem \ref{thm-m-to-m}, which possesses its own independent interest. 
Theorem \ref{thm-m-to-m} offers a potential approach to address Conjecture \ref{main-conj-pn} when the total dual VMRT is a hypersurface in the projectivized tangent bundle. 
For precise definitions, we direct the reader to \cite[Section 2.B.]{HLS22} and the references therein (cf.~Section \ref{s:pf-main}).

\begin{theorem}\label{thm-m-to-m}
Let \(f\colon X\to Y\) be a finite surjective  morphism between Fano manifolds of Picard number 1.
Let \(\mathcal{K}\) and \(\mathcal{G}\) be the dominating families of minimal rational curves on \(X\) and \(Y\) whose VMRTs along a general point are not dual defective.
Suppose that \(Y\) is not isomorphic to a projective space, and via the induced rational map \(\mathbb{P}(T_X)\dashrightarrow \mathbb{P}(T_Y)\), the proper transform of the total dual VMRT of \(\mathcal{K}\) is the total dual VMRT of \(\mathcal{G}\). 
Then the \(f\)-image of a general minimal rational curve of \(X\) is a general minimal rational curve of \(Y\). 
If moreover the VMRT along a general point  of \(X\) is of positive dimensional, then \(f\) is an isomorphism. 
\end{theorem}
Before proceeding to the applications, we give some remarks on Theorem \ref{thm-m-to-m}.
\begin{remark}\label{rmk-thm-preserve-vmrt}
\begin{enumerate}
\item In Theorem \ref{thm-m-to-m}, the assumption that the induced rational map preserves the total dual VMRTs is satisfied if the tangent bundle $T_X$ is big (see Proof of Theorem \ref{thm-main}).
\item To prove the first half of the statement, we show that $f$ maps general minimal rational curves to general minimal rational curves and there is an induced generically finite dominant rational map \(\mathcal{K}\dashrightarrow\mathcal{G}\); in particular, the dimensions of VMRTs along a general point are the same. 
\item To prove that \(f\) is indeed an isomorphism, we need that a general minimal rational curve is even birational to its image, which is the case if the VMRT along a general point has positive dimension (see Proof of Theorem \ref{thm-m-to-m}). 
This conclusion would fail if one drops the assumption on the dimension of the VMRT, see e.g. \cite[P. 86-87, Example and the paragraph after it]{IP99}.
\end{enumerate} 
\end{remark}

As one might expect, there are several applications of our main results. 
Regarding positivity, Theorem \ref{thm-main} provides a criterion for proving the non-bigness of the tangent bundles of certain Fano manifolds via finite non-trivial covers. 
In what follows, combining Theorem \ref{thm-main} with other techniques, we explore the bigness of the tangent bundles of complete intersections, del Pezzo manifolds and Mukai manifolds, yielding the results presented in Theorems \ref{thm-complete-intersection}, \ref{thm-del-Pezzo} and Proposition \ref{prop-Mukai-nonbig} below.
Let us begin with smooth complete intersections in a projective space.

\begin{theorem}\label{thm-complete-intersection}
Let \(X\) be a non-linear smooth complete intersection of multi-degree \(\underline{d}=(d_1,\cdots,d_k)\) in a projective space.
Then the tangent bundle \(T_X\) is big if and only if \(X\) is a quadric hypersurface. 
Moreover, suppose that \(X\) is very general in its deformation family. 
Then \(T_X\) is pseudo-effective if and only if \(\underline{d}=(2)\) or \(\underline{d}=(2,2)\).
\end{theorem}


Recall that when \(X\) is a smooth hypersurface of degree \(d\), the joint paper of H\"oring, Liu and the first author shows that the tangent bundle \(T_X\) is pseudo-effective if and only if \(d\leq 2\) (see \cite[Theorem 1.4]{HLS22}); when \(X\) is a smooth complete intersection of two quadrics, \cite{BEHLV23} verifies that the tangent bundle \(T_X\) is \(\mathbb{Q}\)-effective but not big; when $X$ is a smooth Fano complete intersection of dimension at least 3 and of Fano index 1 or 2, \cite[Theorem 1.1]{HL23} implies that $T_X$ is not big. 
We refer the reader to Theorem \ref{thm-intersection-quadrics} written by Liu, which shows that a general smooth finite cover over a general complete intersection of two quadrics cannot have a pseudo-effective tangent bundle.

Now, let's consider del Pezzo manifolds. 
A smooth Fano variety \(X\) of dimension \(n\) is called a \textit{del Pezzo} manifold, if there exists an integral ample divisor \(H_X\) such that the anti-canonical divisor \(-K_X\) is linearly equivalent to \((n-1)H_X\); the self-intersection \(d(X)\coloneqq H_X^n\) is called the \textit{degree} of \(X\).
Applying Theorem \ref{thm-main}, we prove the non-bigness of del Pezzo manifolds of lower degree, leading to the following result.

\begin{theorem}\label{thm-del-Pezzo}
    Let $X$ be a del Pezzo manifold of degree $d$. 
    Then the tangent bundle $T_X$ is big if and only if $d\geq 5$.
\end{theorem}
The essential cases we deal with are when $d=1,2$. Indeed, if one of the following conditions holds: (1) \(\dim(X)\leq 3\); (2) \(d=3\); or (3) \(d\geq 6\), Theorem \ref{thm-del-Pezzo} has been proved in  \cite[Theorems 1.2 and 1.5]{HLS22} (cf.~\cite{HL23}); if \(d=5\) and \(\dim(X)\geq 4\), it has been proved in \cite[Proposition 3.15 and Example 4.5]{Liu23}; if \(d=4\), it has been proved in \cite{BEHLV23}.
However, the pseudo-effectiveness of the tangent bundle in the case \(d\leq 2\) is still unknown. 

In addition to del Pezzo manifolds, we also study the bigness of tangent bundle of  \textit{Mukai manifolds}, i.e., a smooth Fano variety \(X\) of dimension \(n\) whose anti-canonical divisor \(-K_X\) is linearly equivalent to \((n-2)H_X\) for some integral ample divisor \(H_X\); the \textit{genus} \(g(X)\) is defined as \(\frac{1}{2}H_X^n+1\). 
It is known that if \(X\) is of Picard number 1, then \(2\leq g(X)\leq 12\) and \(g(X)\neq 11\) hold.
\begin{proposition}\label{prop-Mukai-nonbig}
Let $X$ be a Mukai manifold of Picard number 1 and genus $g(X)$. 
If $g(X)\leq 5$, or \(g(X)=6\) and \(\dim(X)\neq 6\), then the tangent bundle $T_X$ is not big. 
\end{proposition}

Recently, an interesting result \cite{KT23} due to Kawakami and Totaro showed that if the Fano manifold \(X\) of Picard number 1 admits a non-isomorphic surjective endomorphism, then it satisfies the \textit{Bott vanishing}; in particular, such \(X\) is \textit{locally rigid}, i.e., \(H^1(X,T_X)=0\).
 This fact, along with Theorem \ref{thm-main} and previous results, allows us to affirmatively answer Conjecture \ref{main-conj-pn} for Mukai manifolds and smooth complete intersections in a projective space. This serves as another application of Theorem \ref{thm-main} from the dynamical viewpoint.
\begin{proposition}\label{prop-mukai-endo}
Let \(X\) be a Mukai manifold of Picard number 1 or a non-linear smooth complete intersection of dimension \(\geq 3\).
Then \(X\) admits no non-isomorphic surjective endomorphism.
\end{proposition}

\subsubsection*{\textbf{\textup{Acknowledgments}}}
The authors would like to express their gratitude to Jun-Muk Hwang for suggesting this project and for many inspiring discussions.
The authors would also like to thank Jie Liu for numerous valuable discussions and suggestions, especially pointing out Example \ref{exa-spinor} and providing the appendix to us, without which they could not verify  Theorem \ref{thm-complete-intersection} for the pseudo-effectiveness of tangent bundles. The authors thank the referee for the very careful reading and the suggestions to improve the paper.  
Special thanks of the second author to IMS at NUS for the warm hospitality and kind support.
Both authors were supported by the Institute for Basic Science (IBS-R032-D1-2023-a00). 

\subsubsection*{\textbf{\textup{Data availability}}}
This manuscript has no associated data.

\subsubsection*{\textbf{\textup{Conflict of interest}}}
On behalf of all authors, the corresponding author states that there is no conflict of interest.

\section{Proofs of Theorems \ref{thm-main} and \ref{thm-m-to-m}}\label{s:pf-main}
In this section, after briefly reviewing the theory of minimal rational curves, we prove Theorems \ref{thm-main} and \ref{thm-m-to-m}. 
For the standard notion and terminology, we refer to \cite{Haw01}.

Let \(X\) be a smooth projective variety.
Denote by \(\textup{RatCurves}^n(X)\) the normalization of the open subset of \(\textup{Chow}(X)\) parametrizing integral rational curves. 
A \textit{dominating family of minimal rational curves} \(\mathcal{K}\) is an irreducible component of \(\textup{RatCurves}^n(X)\) such that the locus \(\text{Locus}(\mathcal{K})\) of \(X\) swept out by curves from \(\mathcal{K}\) is dense in \(X\) and for a general point \(x\in \text{Locus}(\mathcal{K})\) the closed subset \(\mathcal{K}_x\subseteq\mathcal{K}\) parametrizing curves through \(x\) is proper.
It is known that any uniruled projective manifold carries a dominating family of minimal rational curves.
A general element \([\ell]\in \mathcal{K}\) is called a \textit{standard rational curve}, i.e., there exists a non-negative integer \(p\) such that 
\[
f^*T_X\cong\mathcal{O}_{\mathbb{P}^1}(2)\oplus\mathcal{O}_{\mathbb{P}^1}(1)^{\oplus p}\oplus\mathcal{O}_{\mathbb{P}^1}^{\oplus(\dim(X)-p-1)}
\]
where \(f\colon\mathbb{P}^1\to \ell\) is the normalization.
Let \(x\in X\) be a general point, and \(\mathcal{K}_x^n\) the normalization of \(\mathcal{K}_x\), which is a finite union of smooth projective varieties of dimension \(p\).
There is a tangent map \(\tau_x\colon \mathcal{K}_x^n\dashrightarrow\mathbb{P}(\Omega_{X,x})\) sending a curve that is smooth at \(x\) to its tangent direction at \(x\).
Let \(\mathcal{C}_x\) be the image of \(\tau_x\), which is called the \textit{variety of minimal rational tangents} (VMRT for short) at \(x\) associated to \(\mathcal{K}\).
By \cite{Keb02} and \cite{HM04}, the map \(\tau_x\) is the normalization morphism of \(\mathcal{C}_x\) and hence the dimension of \(\mathcal{C}_x\) is equal to the dimension of \(\mathcal{K}_x\)  which is exactly \(p\).

For a standard rational curve \([\ell]\in\mathcal{K}\), a minimal section \(\overline{\ell}\) of \(\mathbb{P}(T_X)\) over the curve \(\ell\) is a section which is given by a surjection \(f^*T_X\to\mathcal{O}_{\mathbb{P}^1}\).
If \(X\) is not a projective space, then the main theorem of \cite{CMSB02} shows that such a minimal section \(\overline{\ell}\) always exists; clearly, we have \(\mathcal{O}_{\mathbb{P}(T_X)}(1)\cdot \overline{\ell}=0\).
The \textit{total dual VMRT} of \(\mathcal{K}\) is defined as
\[
\check{\mathcal{C}}\coloneqq\overline{\bigcup_{\textup{general}\ [\ell]\in\mathcal{K}}\overline{\ell}}^{\textup{Zar}}\subseteq\mathbb{P}(T_X)
\]
where the union is taken over all minimal sections over all standard rational curves in \(\mathcal{K}\).
Note that \(\check{\mathcal{C}}\subseteq\mathbb{P}(T_X)\) is an irreducible projective variety dominating \(X\).
Let \(\check{\mathcal{C}}_x\) be the fibre of \(\check{\mathcal{C}}\to X\) at \(x\in X\).
It follows from \cite[Proposition 5.14]{MOS15} that \(\check{\mathcal{C}}_x\) is the dual variety of \(\mathcal{C}_x\) and hence the total dual VMRT \(\check{\mathcal{C}}\subseteq\mathbb{P}(T_X)\) is a prime divisor 
if and only if for a general point \(x\in X\), \(\mathcal{C}_x\subseteq\mathbb{P}(\Omega_{X,x})\) is not dual defective. 
We refer to \cite{OCW16}, \cite[Section 2.B.]{HLS22} and \cite[Section 3.A.]{FL22} for more information involved.  

We stick to the notations and assumptions below in the remaining part of this section.
\begin{notation}\label{notation}
Let \(f\colon X\to Y\) be a non-isomorphic surjective morphism between Fano manifolds of Picard number 1. 
Let \(T_X\) and \(T_Y\) be the tangent bundles of \(X\) and \(Y\) respectively. 
Assume that \(Y\) is not isomorphic to a projective space. 
Then \(X\) is not isomorphic to a projective space, either; see \cite[Theorem 4.1]{Laz84}.
\begin{enumerate}
\item Dualizing the right exact sequence \(f^*\Omega_{Y}\to\Omega_X\to \Omega_{X/Y}\to 0\) and noting that \(\Omega_{X/Y}\) is a torsion sheaf, we obtain a left exact sequence \(0\to T_X\to f^*T_Y\) which is indeed an isomorphism away from the ramification divisor \(R_f\). 
\item By (1), we obtain the following commutative diagram associated with the injection \(0\to T_X\to f^*T_Y\) such that the induced map \(\mathbb{P}(f^*T_Y)\dashrightarrow \mathbb{P}(T_X)\) is an isomorphism away from the inverse image of \(R_f\) and hence birational.
Let \(\Gamma\) be its graph. 
\[
\xymatrix{&&\Gamma\ar[dl]_{\beta}\ar[dr]^{\alpha}&\\
\mathbb{P}(T_Y)\ar[d]^{\phi}&\mathbb{P}(f^*T_Y)\ar@{-->}[rr]\ar[d]^\varphi\ar[l]_{\widetilde{f}}&&\mathbb{P}(T_X)\ar[d]^\tau\\
Y&X\ar@{=}[rr]\ar[l]_{f}&&X
}
\]
\item Denote by \(\xi\) (resp. \(\eta\)) the tautological line bundle of \(\mathbb{P}(T_X)\) (resp. \(\mathbb{P}(T_Y)\)).
Let \(\widetilde{\eta}\) be the pullback \(\widetilde{f}^*\eta\) which is the tautological line bundle of \(\mathbb{P}(f^*T_Y)\).
\item Let \(\mathcal{K}\) (resp. \(\mathcal{G}\)) be a dominating family of minimal rational curves on \(X\) (resp. on \(Y\)). 
Assume that both VMRTs along a general point are not dual defective. 
\item Denote by \(D_X\subseteq \mathbb{P}(T_X)\) and \(D_Y\subseteq\mathbb{P}(T_Y)\) the total dual VMRTs of \(\mathcal{K}\) and \(\mathcal{G}\)  respectively. 
By our assumption, both \(D_X\) and \(D_Y\) are irreducible hypersurfaces.
\item Let \(H_X\) be the ample generator of the Picard group \(\textup{Pic}(X)\).  
\end{enumerate}
\end{notation}

First, we prove Theorem \ref{thm-m-to-m}. 
For the convenience of the proof, we also recall the following notion in \cite[Definition 2.1]{Hwa15}. 
Let \(M\) be a complex manifold equipped with a closed holomorphic 2-form \(\omega\).
For a point \(z\in M\), let
\[
\text{Null}_z(M)\coloneqq\{u\in T_z(M)~|~\omega(u,v)=0~\text{for all }v\in T_z(M)\}.
\]
This defines a distribution, called the \textit{null distribution} on a Zariski open subset of \(M\).
If \(\text{Null}_z(M)=0\) holds for every point \(z\in M\), then \(\omega\) is a \textit{symplectic form} and \(M\) is a symplectic manifold. 
Now let \((M,\omega)\) be a symplectic manifold equipped with a non-degenerate symplectic 2-form \(\omega\).
Given an irreducible subvariety \(Z\subseteq M\) with the smooth locus of \(Z\) denoted by \(Z_\text{sm}\), we consider the restriction \(\omega|_{Z_{\text{sm}}}\).
The rank of the null distribution of \(\omega|_{Z_{\text{sm}}}\) is no more than the codimension \(\text{codim}_MZ\) and if the equality holds, then we say that \(Z\) is \textit{coisotropic}.
The null distribution on \(Z\) defines a foliation on a Zariski open subset of \(Z\) which we call the \textit{null foliation} of \(\omega\) on \(Z\).

\begin{proof}[Proof of Theorem \ref{thm-m-to-m}]
We consider the following rational map \(\Phi\) between the cotangent spaces $T^*X$ and $T^*Y$:
$$
\begin{aligned}
&T^*X\stackrel{\Phi}\dashrightarrow T^*Y\\
&(s,t)\mapsto (f(s),(df_s^*)^{-1}(t))
\end{aligned}
$$
where $s\in X$ is a point outside the support of the ramification divisor of $f$, $t\in T_s ^*X$ is a cotangent vector and $df_s^*:T^*_{f(s)}Y\to T^*_sX$ is the dual of the differential map. 
We note that for any point \(x\in X\), the cotangent space \(T_x^*X\) is the affine cone over \(\mathbb{P}(T_{X,x})\).

Let $\omega$ be the natural symplectic form on $T^*Y$. 
In the following, by abuse of notation, we identify \(D_X\) and \(D_Y\) with their affine cones. 
For a general point $z\in T^*X$, $d\Phi_z $ is an isomorphism and  $\Phi$ induces a 2-form $\Phi^*(\omega)$ of $T^*X$ defined by the following:
$$
\Phi^*(\omega)(u,v)\coloneqq\omega(d\Phi_z(u), d\Phi_z(v)),
$$  
where $u,v\in T_{z}(T^*X)$ and $\Phi^*(\omega)$ is a symplectic form on a open subset of $T^*X$.

Suppose that $C$ is a general leaf of the null foliation of $\Phi^*(\omega)$ on $D_X$. 
For a general $z\in C$, we have $\Phi^*(\omega)(u,v)=0$ for arbitrary $v\in T_z C$ and $u\in T_z D_X$.  
Consider the image $\Phi (C)$. 
By our assumption, the total dual VMRT \(D_X\) is mapped onto the total dual VMRT \(D_Y\) along the rational map \(\Phi\); hence \(\Phi(C)\) is contained in \(D_Y\).
Given any $v'\in T_{\Phi(z)}\Phi(C)$ and any \(u'\in T_{\Phi(z)}D_Y\), we obtain that
$$
\omega(u',v')=\Phi^*(\omega)(d\Phi_z^{-1}(u'), d\Phi_z^{-1}(v'))=0,
$$
noting that $d\Phi_z^{-1}(u')\in T_z D_X$ and $d\Phi_z^{-1}(v')\in T_z C$.  
Therefore, $\Phi(C)$ is a leaf of the null foliation of $\omega$ on  $D_Y$. 
By \cite[Proposition 2.4]{Hwa15}, both \(D_X\) and \(D_Y\) are coisotropic and 
the closure of $C$ and the closure of $\Phi(C)$  are minimal sections over minimal rational curves; moreover, a general minimal section of \(\tau\) (resp. \(\phi\)) can be realized as the closure of a leaf of the null foliation of $\Phi^*(\omega)$ (resp. $\omega$) on $D_X$ (resp. \(D_Y\)). 

Let  \(\mathcal{M}_X\subseteq \text{Chow}(\mathbb{P}(T_X))\) and \(\mathcal{M}_Y\subseteq\text{Chow}(\mathbb{P}(T_Y))\) be the families of minimal sections of \(\tau\) and \(\phi\), respectively.
Then we have the following commutative diagram
\[
\xymatrix{
\text{Chow}(\mathbb{P}(T_X))\ar@{-->}[r]\ar[d]_{\tau_*}&\text{Chow}(\mathbb{P}(T_Y))\ar[d]^{\phi_*}\\
\text{Chow}(X)\ar@{-->}[r]&\text{Chow}(Y)
}
\]
Since the proper transform of \(D_X\) is \(D_Y\) and a general minimal section of \(\tau\) (resp. \(\phi\)) is the closure of a leaf of the  null foliation on \(D_X\) (resp. \(D_Y\)), by the above argument, a general minimal section of \(\tau\) is mapped to a general minimal section of \(\phi\); in particular, \(\mathcal{M}_X\) is sent to \(\mathcal{M}_Y\) via the first horizontal map. 
Therefore, we obtain the induced map \(\mathcal{K}\dashrightarrow\mathcal{G}\) via the second horizontal map which is also dominant. 
In particular, $f$ maps a general minimal rational curve \(\ell\) of \(\mathcal{K}\) to a general minimal rational curve \(\ell'\) of \(\mathcal{G}\) and the first half of our statement is completed. 

Assume that the normal bundle \(\mathcal{O}_{\ell/Y}\) (resp.~\(\mathcal{O}_{\ell'/Y}\)) is of the form \(\mathcal{O}_{\ell}(1)^{\oplus p}\oplus \mathcal{O}_{\ell}^{\oplus (\dim(X)-1-p)}\) (resp.~\(\mathcal{O}_{\ell'}(1)^{\oplus p'}\oplus \mathcal{O}_{\ell'}^{\oplus (\dim(Y)-1-p')}\)). 
Let \(\mathcal{U}\) and \(\mathcal{V}\) be the universal families of \(\mathcal{K}\) and \(\mathcal{G}\) respectively.
Then we obtain the following commutative diagram
\[
\xymatrix{
\mathcal{K}\ar@{-->}[d]&\mathcal{U}\ar[l]_{\rho_1}\ar[r]^{e_1}&X\ar[d]^f\\
\mathcal{G}&\mathcal{V}\ar[l]_{\rho_2}\ar[r]^{e_2}&Y
}
\]
On the one hand, as \(\mathcal{K}\dashrightarrow\mathcal{G}\) is dominant, for a general point \(x\in X\), the induced map
\[
\mathcal{K}_x=\rho_1(e_1^{-1}(x))\dashrightarrow\mathcal{G}_{f(x)}=\rho_2(e_2^{-1}(f(x)))
\]
is also dominant. 
Hence, \(p=\dim(\mathcal{K}_x)\geq\dim(\mathcal{G}_{f(x)})=p'\).
On the other hand, since \(f\) is finite and the inverse image \(f^{-1}(\mathcal{G})\) contains \(\mathcal{K}\) as an irreducible component, it follows that
\[
p'=\dim(f^{-1}(\mathcal{G})_x)\geq\dim(\mathcal{K}_x)=p.
\]
This implies that \(p=p'\).

In the following, we assume that \(p\geq 1\). 
We claim that, for a general point \(x\in X\) away from the ramification divisor, there exists a general element \([\ell]\in\mathcal{K}_x\) which is birational to its image \(\ell'\coloneqq f(\ell)\). 
Indeed, as the normal bundle \(\mathcal{N}_{\ell/X}\) cannot have sections vanishing along two distinct points, the deformation family of minimal rational curves with two points fixed is 0-dimensional. Hence there exists a minimal rational curve $\ell$ passing through the point $x$ but away from \(\textup{Supp}\,(f^{-1}(f(x)))\backslash\{x\}\) so  that \(\deg(f|_{\ell})=1\) (cf.~\cite[Lemma 2.1]{Haw01}).  
So our claim is thus proved. 

From the following normal bundle sequences
\[
0\to T_\ell=\mathcal{O}_\ell(2)\to T_X|_\ell\to\mathcal{N}_{\ell/X}=\mathcal{O}_\ell(1)^{\oplus p}\oplus\mathcal{O}_{\ell}^{\oplus(\dim(X)-1-p)}\to 0, \textup{ and}
\]
\[
0\to T_{\ell'}=\mathcal{O}_{\ell'}(2)\to T_Y|_{\ell'}\to\mathcal{N}_{\ell'/Y}=\mathcal{O}_{\ell'}(1)^{\oplus p'}\oplus\mathcal{O}_{\ell'}^{\oplus(\dim(Y)-1-p')}\to 0,
\]
we obtain that 
\[K_X\cdot\ell=-\det(T_X|_\ell)=-\det(T_Y|_{\ell'})=K_Y\cdot\ell'=K_Y\cdot f_*(\ell).
\]
However, this leads to a contradiction to the ampleness of the ramification divisor \(R=K_X-f^*K_Y\), noting that \(Y\) is simply connected and \(X\) is of Picard number 1.
Theorem \ref{thm-m-to-m} is thus proved.
\end{proof}

Second, we prove Theorem \ref{thm-main}.
\begin{proof}[Proof of Theorem \ref{thm-main}]
As \(\xi\) is big, it follows from \cite[Theorem 3.4]{FL22} (cf.~ \cite{HLS22}) that {\color{blue}\((*)\)} \(D_X\equiv a\xi-b\tau^*H_X\) holds where \(a>0\) is the codegree of the VMRT \(\mathcal{C}_x\) at a general point and \(b>0\) is an integer.
Since the total dual VMRT \(D_X\) (as a prime divisor in \(\mathbb{P}(T_X)\) by our assumption) \(D_X\) is covered by minimal sections \(\overline{\ell}\subseteq\mathbb{P}(T_X)\) of \(\mathcal{K}\) such that \(\xi\cdot \overline{\ell}=0\), we have \(D_X\cdot \overline{\ell}<0\).
Hence, \(D_X|_{D_X}\) is not pseudo-effective and thus the irreducible divisor \(D_X\) is not big; in particular, \(D_X\)  is extremal in the pseudo-effective cone \(\text{PE}(\mathbb{P}(T_X))=\mathbb{R}_{\geq 0}[D_X]+\mathbb{R}_{\geq 0}[\tau^*H_X]\) (see \cite[Theorem 3.4 (3)]{FL22}).  
In addition, as \(\mathbb{P}(T_X)\) is simply connected, the numerical equivalence of integral Cartier divisors {\color{blue}\((*)\)} is indeed a linear equivalence. 
Let $D_X'\coloneqq\beta_*(\alpha_*^{-1}(D_X))\subseteq\mathbb{P}(f^*T_Y)$ be the proper transform of $D_X$ along the birational map $\mathbb{P}(f^*T_Y)\dashrightarrow \mathbb{P}(T_X)$. 

To apply Theorem \ref{thm-m-to-m}, we are left to show the equality \(D_X'=\widetilde{f}^*D_Y\). 
By the natural injection \(0\to T_X\to f^*T_Y\), we have the induced injection
\[
0\to H^0(X,\Sym^aT_X\otimes\mathcal{O}_X(-bH_X))\to H^0(X,\Sym^a(f^*T_Y)\otimes \mathcal{O}_X(-bH_X)).
\]
On the one hand, as both \(\alpha\) and \(\beta\) are birational with connected fibres, we have the induced injection
\[
0\to H^0(\Gamma,a\alpha^*\xi-b\alpha^*\tau^*H_X)\to H^0(\Gamma,a\beta^*\widetilde{\eta}-b\beta^*\varphi^*H_X).
\]
On the other hand, by the linear equivalence \(D_X\sim a\xi-b\tau^*H_X\) and the fact that \(T_X\to f^*T_Y\) is an isomorphism away from the ramification divisor \(R_f\) of \(f\), there exists some \(m\geq 0\) such that 
\[
a\widetilde{\eta}+\varphi^*(-bH_X)\sim D_X'+m\varphi^*H_X,
\]
noting that the induced birational map \(\alpha\circ\beta^{-1}\) is then an isomorphism away from the pullback \(\varphi^*(\textup{Supp}\,R_f)\) of the support of \(R_f\). 
Therefore, we have \(D_X'\sim a\widetilde{\eta}-(b+m)\varphi^*H_X\).
Since both \(\alpha\) and \(\beta\) are birational and isomorphisms over the \'etale locus of \(f\), together with \(D_X\) being dominant over \(X\), it follows that \(D_X'\) is a prime divisor.
As the total dual VMRT \(D_Y\) is covered by minimal sections \(\overline{c}\) such that \(\eta\cdot \overline{c}=0\), by the projection formula, its pullback \(\widetilde{f}^*D_Y\) is covered by curves \(\overline{c}'\) such that \(\widetilde{\eta}\cdot\overline{c}'=0\).
In particular,
\[
\overline{c}'\cdot D_X'=\overline{c}'\cdot (a\widetilde{\eta}-(b+m)\varphi^*H_X)<0,
\]
where we used the fact that \(b>0\) due to the bigness of  \(T_X\) (cf.~\cite[Theorem 3.4]{FL22}). 
Therefore, \(\overline{c}'\subseteq D_X'\) and hence \(\textup{Supp}\,\widetilde{f}^*D_Y\subseteq\textup{Supp}\,D_X'\).
By the irreducibility of \(D_X'\), the desired equality is proved. So if the VMRT at a general point of $X$ has positive dimension, then the result follows from Theorem \ref{thm-m-to-m}.

Now assume that the dimension of VMRT at a general point of $X$ is 0. Then the dimension of VMRT at a general point of $Y$ is also 0 (see Remark \ref{rmk-thm-preserve-vmrt}). 
By Lemma \ref{positivity-finite-morphism}, $T_Y$ is also big. 
Then it follows from \cite[Theorem 1.1]{HL23} that both $X$ and $Y$ are isomorphic to quintic del Pezzo threefolds.  
As there is a unique isomorphic class of quintic del Pezzo threefolds (see e.g. \cite[Remark 3.8]{BFM20}), we obtain from \cite[Corollary 3]{HM03} that any surjective morphism from $X$ to $Y$ is an isomorphism.
\end{proof}

\section{Examples and Proof of Proposition \ref{prop-mukai-endo}}
In this section, after collecting some typical examples of Theorem \ref{thm-main}, we prove Proposition \ref{prop-mukai-endo}.
\begin{example}[Rational homogeneous spaces of Picard number 1]\label{exa-rational-homogeneous}
By the generic finiteness of the Springer map, the tangent bundle of a rational homogeneous space is always big (cf.~\cite[Corollary 4.4]{GW20}, \cite[Proposition 6.2]{Sha23} and \cite[Example 3.9.2]{Liu23}). 
Besides, the VMRT of a rational homogeneous space of Picard number 1 can be fully determined by \cite[Theorem 4.8]{LM03}. 
We refer the reader to \cite[Appendix A]{FL22} for a full classification of rational homogeneous spaces of Picard number 1 whose VMRTs are not dual defective.
\end{example}

\begin{example}[Quintic del Pezzo fourfolds]\label{e:quintic}
Let \(X\) be a del Pezzo manifold of Picard number 1 and  dimension \(n\).
Let \(H_X\) be the ample generator of \(\text{Pic}(X)\) such that the degree \(d(X)\coloneqq H_X^n\) is 5.
Then it is known that \(3\leq n\leq 6\) and the tangent bundle \(T_X\) is big (see \cite[Theorem 1.5]{HLS22} and \cite[Proposition 3.15 and Example 4.5]{Liu23}; cf. \cite[Theorem 1.1]{HL23}).
If \(n\leq 5\), then the VMRT at a general point is not dual defective, and hence such \(X\) does not admit any non-isomorphic surjective endomorphism by applying Theorem \ref{thm-main} (cf.~\cite[Theorem 1.4]{HN11}).
If \(n=6\), then \(X\) is homogeneous and hence there is no non-isomorphic surjective endomorphism, either (see \cite{PS89}); however, in this case, the VMRT along a general point is \(\mathbb{P}^1\times\mathbb{P}^2\subseteq\mathbb{P}^5\) which is dual defective.
\end{example}

The following is a typical example that is not almost homogeneous but has a big tangent bundle with the VMRT along a general point not dual defective. 
\begin{example}\label{exa-spinor}
Let \(X\) be a smooth linear section of a 10-dimensional spinor variety of codimension 3.
It is known that there are four isomorphism classes and the general one is not almost homogeneous (see \cite[Proposition 33, p. 126]{SK77}). 
On the other hand, the tangent bundle \(T_X\) is big and the VMRT along a general point is not dual defective; see \cite[Corollary 4.18]{Liu23}.
By Theorem \ref{thm-main}, such \(X\) does not admit any non-isomorphic surjective endomorphism.
\end{example}

\begin{proof}[Proof of Proposition \ref{prop-mukai-endo}]
Suppose to the contrary that \(X\) admits a non-isomorphic surjective endomorphism.
By \cite[Proposition 2.7]{KT23}, such \(X\) satisfies Bott vanishing and then \(X\) is locally rigid.
However, as proved in \cite{BFM20}, no smooth complete intersection in a given projective space is locally rigid except for the projective space and the hyperquadric, while the latter one admits no endomorphism by \cite[Proposition 8]{PS89}.

In the following, we assume that \(X\) is a Mukai manifold of Picard number 1, of dimension \(n\), and of genus \(g(X)\). 
We refer the reader to \cite{Muk89} for the classification of Mukai manifolds.
We may further assume that \(n\geq 5\) (cf.~\cite{ARV99}, \cite{HM03} and \cite[Corollary 5.2]{SZ22}). 
By \cite[Theorem 3.1 and Proposition 3.6]{SZ22} and \cite{Bea01} (cf.~\cite[Remark 3.8]{SZ22}), we may assume that \(g(X)\geq 4\) and hence the ample generator \(H_X\) of the Picard group \(\text{Pic}(X)\) is very ample (see e.g. \cite[Theorem 5.2.1 and the paragraph after it]{IP99}). 
As \(X\) is locally rigid, by \cite[Theorem 1.1]{BFM20} and \cite[Lemma 1.10]{Per16}, such \(X\) can only be one of the following:
\begin{enumerate}
\item a general codimension \(k\) (with \(k\leq 3\))  linear section of a 10-dimensional spinor variety;
\item a general hyperplane section of \(\text{Gr}(2,6)\);
\item a general hyperplane section of the symplectic Grassmannian \(\text{SG}(3,6)\).
\end{enumerate}
On the other hand, it follows from \cite[Theorem 1.2, Propositions 4.7 and 4.8]{BFM20} that all of the above varieties except for Case (1) with \(k=3\) are almost homogeneous, and hence they admit no non-isomorphic surjective endomorphisms (see \cite[Theorem 1.4]{HN11}).
The only case left when \(X\) is a smooth linear section of a 10-dimensional spinor variety of codimension 3 has been excluded by our Theorem \ref{thm-main} and \cite[Corollary 4.18]{Liu23} (cf.~Example \ref{exa-spinor}).
\end{proof}


\section{Bigness of tangent bundles of certain projective manifolds}\label{s:ex}
In this section, using Theorem \ref{thm-main} and developing some other techniques, we study the bigness of tangent bundles of smooth complete intersections, del Pezzo manifolds, and Mukai manifolds. 

\subsection{Smooth complete intersections, Proof of Theorem \ref{thm-complete-intersection}}


\begin{lemma}\label{positivity-finite-morphism}
    Let $f\colon X\to Y$ be a generically finite surjective morphism between smooth projective varieties. 
    If the tangent bundle \(T_X\) is big (resp. pseudo-effective), then the tangent bundle $T_Y$ is also big (resp. pseudo-effective).
\end{lemma}
\begin{proof}
    Consider the surjective morphism $\mathbb{P}(f^*T_Y)\to \mathbb{P}(T_Y)$ induced by $f$. 
    Then the tautological line bundles of $\mathbb{P}(f^*T_Y)$ and $\mathbb{P}(T_Y)$ have the same Kodaira dimension by \cite[Theorem 5.13]{Uen75}. 
    Hence, $T_Y$ is big if and only if $f^* T_Y$ is big. 
    Since \(f\) is generically surjective and \(T_X\) is locally free, there is a natural injection $0\to T_X\to f^*T_Y$, and thus the bigness of $T_X$ implies the bigness of $T_Y$.

    Now suppose that $T_Y$ is not pseudo-effective. 
    Let $\eta$ (resp. $\widetilde{\eta}$) be the tautological line bundle of $\mathbb{P}(T_Y)$ (resp. $\mathbb{P}(f^*T_Y)$) and let $\pi\colon  \mathbb{P}(T_Y)\to Y$ and $\widetilde{\pi}\colon \mathbb{P}(f^*T_Y)\to X$ be the natural projections. 
    Let \(A\) be any ample divisor on $Y$.
    Then $\eta+\frac{1}{n}\pi^* A$ is not $\mathbb{Q}$-effective for any sufficiently large integer $n$.
    Hence, $\widetilde{\eta}+\frac{1}{n}\widetilde{\pi}^*f^*A$ is not $\mathbb{Q}$-effective for any sufficiently large integer $n$  (cf.~\cite[Theorem 5.13]{Uen75}). 
    Applying \cite[Lemma 2.2]{HLS22} and the injection $0\to T_X\to f^*T_Y$, we see that $T_X$ is not pseudo-effective, either.
\end{proof}

\begin{lemma}\label{l:trivial-VMRT}
Let \(X\) be a smooth non-linear Fano complete intersection of dimension \(\geq 3\). 
Then the VMRT is not dual defective along a general point.
\end{lemma}
\begin{proof}
If the Fano index \(\geq 2\), then \(X\) is covered by lines; in particular, the VMRT is not dual defective by \cite[Example 1.4.2]{Haw01} and \cite[Theorem 5.11]{Tev05}.
Suppose that the Fano index is 1. 
Then it is known that \(X\) is covered by conics; see \cite{CM80}.
In this case, the VMRT along a general point is 0-dimensional which is not dual defective, either.
\end{proof}

\begin{proof}[Proof of Theorem \ref{thm-complete-intersection}]
Without loss of generality, we may assume that 
\(X\subseteq\mathbb{P}^{n+k}\) is a smooth complete intersection of dimension \(n\geq 2\) and of multi-degree \((d_1,\cdots,d_k)\) with \(d_i\geq 2\) for any \(1\leq i\leq k\).  
First, we consider the bigness of the tangent bundle.
If $T_X$ is big, then $X$ is uniruled by Miyaoka’s generic 
semipositivity theorem (cf.~\cite[Proposition 4.6]{GW20}). 
Since $X$ is a complete intersection, $X$ is Fano. 
If $\dim(X)=2$, then the result follows from \cite[Theorem 1.2]{HLS22}. 
So in the following, we assume $\dim(X)\geq 3$ and hence $X$ has Picard number 1 by the Lefschetz theorem. 
If \(k=1\), then \(X\) is a smooth hypersurface of degree \(d\geq 2\); hence \(T_X\) is big if and only if \(d=2\); see \cite[Theorem 1.4]{HLS22}.  
Let us further assume that \(k\geq 2\).  
Consider the following projection map:
\begin{align*}
\mathbb{P}^{n+k}\backslash\{p\coloneqq[0:\cdots:0:1]\}&\xrightarrow{\pi}\mathbb{P}^{n+k-1}\\
[x_0:\cdots:x_{n+k-1}:x_{n+k}]&\mapsto[x_0:\cdots:x_{n+k-1}].
\end{align*}
Let $Y\subseteq \mathbb{P}^{n+k-1}$ be a general smooth complete intersection of multi-degree $(d_1,...,d_{k-1})$.
Take \(Y'\subseteq\mathbb{P}^{n+k}\) to be the projective cone over \(Y\), in other words, \(Y'\) is defined by the closure of \(\pi^{-1}(Y)\).
It is clear that the ideal sheaf \(\mathcal{I}_{Y'}\) of \(Y'\) is generated by the same polynomials as that of \(Y\).
By Bertini's theorem, one can take a sufficiently general degree \(d_k\) hypersurface \(H_k\) in \(\mathbb{P}^{n+k}\) such that the scheme-theoretic intersection \(X'\coloneqq Y'\cap H_k\) is a smooth complete intersection of multi-degree \((d_1,\cdots,d_k)\). 
By construction, \(X'\to Y\) is a finite cover of degree \(d_k\), noting that \(X'\) does not pass through the singular vertex \(p\). 
By Lemma \ref{l:trivial-VMRT}, the VMRT of  \(X'\) and the VMRT of \(Y\) at general points are not dual defective.
Applying Theorem \ref{thm-main}, we obtain that the tangent bundle of $X'$ is not big.
As  $X'$ and $X$ are both smooth complete intersections of the same multi-degree, they can be deformed to each other. 
By \cite[Proposition 4.16]{Liu23}, the tangent bundle of \(X\) is not big. 

Second, we consider the pseudo-effectiveness of the tangent bundle.
By Semicontinuity theorem \cite[Chapter III, Theorem 12.8]{Har77} and \cite[Lemma 2.2]{HLS22}, it is sufficient to show that there exists a non-degenerate smooth complete intersection of multi-degree $(d_1,d_2,..,d_k)$ such that its tangent bundle is not pseudo-effective.  We may assume that $d_1\geq \cdots\geq d_k$. If $k=1$, then the result follows from \cite[Theorem 1.4]{HLS22}. Suppose that $k\geq 2$. 
Then by the first half of the proof, there exists a smooth complete intersection $X$ of multi-degree $(d_1,...,d_k)$ and a finite morphism from $X$ to a smooth hypersurface $Y$ of degree $d_1$. 
If $d_1\geq 3$. 
By \cite[Theorem 1.4]{HLS22} again, $T_Y$ is not pseudo-effective, and hence $T_X$ is not pseudo-effective either; see Lemma \ref{positivity-finite-morphism}. 
Suppose finally that $d_1=2$. 
By assumption, we may assume that \(k\geq 3\) and then by the first half of the proof, there exists a general finite double cover from a smooth complete intersection of $k$ quadrics \(X\) onto a general smooth complete intersection of \((k-1)\) quadrics \(Y\). 
We note that here, every double cover is uniquely determined by the \(Y\) and a given branch divisor \(H\in|\mathcal{O}_Y(2)|\), in other words, once we fix such \(Y\) and \(H\), the induced cyclic double cover \(X\) is a complete intersection of \(k\) quadrics; see e.g. \cite[Proposition 15]{Dim86}. 
Hence, in this case, the result follows from the induction, Lemma \ref{positivity-finite-morphism} and Theorem \ref{thm-intersection-quadrics}.
\end{proof}

\subsection{Del Pezzo manifolds, Proof of Theorem \ref{thm-del-Pezzo}}

The following proposition is a bit technical and has its own independent interest especially when we consider the projective dual of a singular variety with its normalization being smooth (cf.~\cite[Proposition 3.1]{Ein86}). 
We refer the reader to \cite[Section 2]{OCW16} for certain similar arguments under a stronger assumption of nodal singularities.
\begin{proposition}\label{p:dual-defect}
Let \(M\) be a smooth quasi-projective variety, \(X\subseteq M\) a smooth projective subvariety, and \(H\) a base-point-free ample Cartier divisor on \(X\).
Suppose that the morphism \(t\colon X\to P\coloneqq\mathbb{P}^N\) induced by the linear system \(|H|\) is birational to its image \(Y\) (in particular, \(X\to Y\) is the normalization). 
Suppose further that there is a smooth morphism \(t'\colon M\to P\) which is a lift of \(t\), i.e., \(t=t'|_X\). 
If the twisted normal bundle \(N_{X/M}\otimes\mathcal{O}_X(-H)\) is ample, then the dual variety \(Y^*\subseteq P^*\) of \(Y\subseteq P\) is a hypersurface.
\end{proposition}

\begin{proof}
Let \(V=H^0(X,H)\) with \(\dim(V)=N+1\).
First, we have the following natural Euler sequence 
\[
0\to \mathcal{O}_P(-1)\to V\to T_P(-1)\to 0.
\]
Pulling it back to \(X\), we obtain the surjection
\[
V\to t^*T_P(-H)\to 0.
\]
Second, we have the following natural short exact sequence of coherent sheaves
\[
0\to T_X(-H)\to t^*T_P(-H)\to \mathcal{N}(-H)\to 0,
\]
where \(\mathcal{N}(-H)\) is not necessarily locally free as the normalization map \(t\) is not necessarily unramified. 
Hence, the composite surjection 
\(V\twoheadrightarrow \mathcal{N}(-H)\) 
gives rise to an injection 
\[W\coloneqq\mathbb{P}_X(\mathcal{N}(-H))\hookrightarrow X\times {\mathbb{P}^N}^*.\]
Note that here \(W\) is merely a projective scheme, possibly very singular. 
On the other hand, as \(t'\colon M\to P\) is smooth, we have the surjective map \(T_{M}\to t'^*T_P\).
Restricting the map to \(X\), we obtain the following commutative diagram of short exact sequences
\[
\xymatrix{
0\ar[r]&T_X(-H)\ar[r]\ar@{=}[d]&T_{M}|_X(-H)\ar@{->>}[d]\ar[r]&N_{X/M}(-H)\ar[r]\ar@{->>}[d]&0\\
0\ar[r]&T_X(-H)\ar[r]&t^*T_P(-H)\ar[r]&\mathcal{N}(-H)\ar[r]&0
}
\]
where the first horizontal sequence is the normal bundle sequence for the smooth projective variety \(X\).
Now that the middle vertical one is a surjection, we obtain that \(N_{X/M}(-H)\to \mathcal{N}(-H)\) is also a surjection.
By the ampleness of \(N_{X/M}(-H)\), we see that the coherent sheaf \(\mathcal{N}(-H)\) is ample in the sense that \(\mathcal{O}_W(1)\) is ample.
Let us consider the following commutative diagram
\[
\xymatrix{
X\times{\mathbb{P}^N}^*\ar[r]&Y\times{\mathbb{P}^N}^*\ar[r]&{\mathbb{P}^N}^*\\
W\ar[r]\ar@{^(->}[u]\ar[d]&W'\ar[r]\ar@{^(->}[u]\ar[d]&Y^*\ar@{^(->}[u]\\
X\ar[r]&Y&
}
\]
where \(W'\) and \(Y^*\) are the images of \(W\) along the composite maps \(W\to X\times{\mathbb{P}^N}^*\to Y\times{\mathbb{P}^N}^*\) and \(W\to X\times{\mathbb{P}^N}^*\to {\mathbb{P}^N}^*\).
By definition, \(W'\) is the conormal variety of \(Y\) which is the closure of \(\mathbb{P}_{U}(\mathcal{N}_{U/P}\otimes\mathcal{O}_U(-1))\) in \(Y\times {\mathbb{P}^N}^*\) where \(U\) is the smooth locus of \(Y\), and \(Y^*\) is the dual variety of \(Y\).
Now that \(\mathcal{O}_W(1)\) (which coincides with the pullback of \(\mathcal{O}_{{\mathbb{P}^N}^*}(1)\)) is ample, we see that \(W\to Y^*\) is a finite (and thus surjective) morphism; in particular, the equality 
\[
\dim(Y^*)=\dim (W)=\dim(X)+\text{rank}(\mathcal{N}(-H))-1=\dim(P)-1=N-1
\]
concludes our proposition.
\end{proof}

\begin{proof}[Proof of Theorem \ref{thm-del-Pezzo}]
Let \(X\) be a del Pezzo manifold of dimension \(n\) such that \(-K_X\sim (n-1)H_X\) where \(H_X\) is an integral ample divisor. 
We may assume that \(n\geq 3\) in view of  \cite[Theorem 1.2]{HLS22}. 
It is known that \(1\leq d\leq 8\).
If  \(d\geq 6\), then \(T_X\) is big (see \cite[Proposition 5.16]{HLS22}). 
If \(d=5\), then it follows from \cite[Theorem 5.4]{HLS22} and \cite[Propositino 3.15 and Example 4.5]{Liu23} that \(T_X\) is big. 
If $d=3$, by \cite[Theorem 1.4]{HLS22}, $T_X$ is not pseudo-effective.
Suppose that \(d=4\). 
Then $X$ is a smooth complete intersection of two hyperquadrics.
By \cite{BEHLV23}, \(T_X\) is \(\mathbb{Q}\)-effective but not big; alternatively, applying Theorem \ref{thm-complete-intersection}, we have \(T_X\) is not big.  
So we are left to consider \(d=1\) or 2.

\par \vskip 0.5pc \noindent
\textbf{The case $d=2$.} 
Then $X$ is a double cover of $\mathbb{P}^n$ branched along a smooth hypersurface of degree 4. 
Its VMRT $\mathcal{C}_x\subset \mathbb{P}(\Omega_{X,x})$ at a general point $x\in X$ is a smooth complete intersection of multi-degree $(3,4)$ by \cite[Theorem 1.1]{HK13} and hence not dual defective by \cite[Theorem 5.11]{Tev05}. Thanks to \cite[Proposition 4.16]{Liu23} and the Semicontinuity theorem, we only need to show a special one in the deformation family of del Pezzo manifolds of degree 2 having the non-big tangent bundle. 
Consider the del Pezzo manifold $X$ of degree 2 defined by the following equation in the weighted projective space $\mathbb{P}(2,1,...,1)$
$$
x_0^2+x_1^4+...+x_{n+1}^4=0.
$$
There is a finite morphism from $X$ to the hyperquadric $Q^n$ defined by the following
$$
\begin{aligned}
X&\longrightarrow Q^n=(x_0^2+x_1^2+...+x_{n+1}^2=0)\subset \mathbb{P}^{n+1}\\
(t_0:t_1:...:t_{n+1})&\mapsto (t_0:t_1^2:...:t_{n+1}^2).
\end{aligned}
$$
Hence $T_X$ is not big by Theorem \ref{thm-main} and Lemma \ref{l:trivial-VMRT}.

\par \vskip 0.5pc \noindent
\textbf{The case $d=1$.} We apply Proposition \ref{p:dual-defect} to show that the VMRT at a general point of $X$ is not dual defective. 
By \cite[Propositions 5.2 and 6.7]{HK15}, the tangent map \(\tau_x\) (which is a normalization morphism due to \cite{Keb02} and \cite{HM04}) along a general point sends a smooth complete intersection \(\mathcal{K}_x\) of multi-degree \((4,5,6)\) in the weighted projective space \(\mathbb{P}\coloneqq\mathbb{P}(2,1^n)\) to a codimension two subvariety \(\mathcal{C}_x\) in \(\mathbb{P}^{n-1}\) such that \(\tau_x\) is the restriction of the natural projection \(\mathbb{P}(2,1^{\oplus n})\backslash\{[1:0:\cdots:0]\}\to \mathbb{P}^{n-1}\). 
Since the twisted normal bundle \(N_{\mathcal{K}_x/\mathbb{P}}(-1)\) of the smooth complete intersection of multi-degree \((4,5,6)\) is isomorphic to \(\mathcal{O}_{\mathcal{K}_x}(3)\oplus\mathcal{O}_{\mathcal{K}_x}(4)\oplus\mathcal{O}_{\mathcal{K}_x}(5)\) which is ample, the VMRT \(\mathcal{C}_x\) is not dual defective by applying Proposition \ref{p:dual-defect}.

Using \cite[Proposition 4.16]{Liu23}, we only need to show a special one in the deformation family of del Pezzo manifolds of degree 1 having the non-big tangent bundle.  
In fact, in \cite[Proposition 4.16]{Liu23}, the assumption of smoothness of the VMRT is not needed if the VMRT of any element in the deformation family at a general point is not dual defective. 
Consider the del Pezzo manifold $X$ of degree 1 defined by the following equation in the weighted projective space $\mathbb{P}(3,2,1,...,1)$
$$
x_0^2+x_1^3+x_2^6+...+x_{n+1}^6=0. 
$$
Let $Y$ be the smooth Fermat hypersurface of degree 6 in the projective space $\mathbb{P}^{n+1}$ defined by $$
x_0^6+x_1^6+x_2^6+...+x_{n+1}^6=0. 
$$
Then there is a finite morphism $g\colon Y\to X$ sending the point $(t_0:t_1:t_2:...:t_{n+1})$ to the point $(t_0^3:t_1^2:t_2:...:t_{n+1})$. Consider the natural injection 
$$0\to g^*\Omega_X\to \Omega_Y.$$
Taking the $(n-1)$-th exterior power of the above exact sequence and then tensoring  with $(n-1)H_Y$, we get an injection:
$$
0\to g^*T_X\to T_Y(3).
$$
Noting that the pushforward $g_*\mathcal{O}_Y$ has a direct summand $\mathcal{O}_X$, we have
$$
H^0(X,\Sym^r T_X)\hookrightarrow H^0(Y,g^*(\Sym^r T_X))\hookrightarrow H^0(Y,\Sym^r (T_Y(3)))\ \ \forall r\geq 1.
$$
Observe that $H^0(Y,\Sym^r (T_Y(3)))=0$ for $r\geq 1$ (see \cite[Theorem 1.3]{HLS22}). 
This implies that $H^0(X,\Sym^r T_X)=0$ for any $r\geq 1$; hence, $T_X$ is not big.
\end{proof}

\subsection{Mukai manifolds, Proof of Proposition \ref{prop-Mukai-nonbig}}

Before proving Proposition \ref{prop-Mukai-nonbig}, we first consider Gushel-Mukai manifolds.
We note that for a Gushel-Mukai sixfold, it admits a double cover over the Grassmannian \(\text{Gr}(2,5)\), the VMRT of which along a general point is self-dual and hence dual defective (cf.~Example \ref{e:quintic}); in particular, we could not apply our Theorem \ref{thm-main}.

\begin{proposition}[Gushel-Mukai manifold]\label{prop-GM}
    Let $X$ be a Gushel-Mukai manifold, i.e., a Mukai manifold of genus 6. 
    If \(\dim(X)\neq 6\), then the tangent bundle $T_X$ is not big.
\end{proposition}
\begin{proof}
    Let \(\text{Gr}(2,5)\) be the Grassmannian of 2-dimensional subspaces of a complex vector space of dimension 5, viewed in \(\mathbb{P}(\wedge^2\mathbb{C}^5)\cong\mathbb{P}^9\) via the Pl\"ucker embedding.
Let \(\text{CGr}(2,5)\subseteq\mathbb{P}(\mathbb{C}\oplus \wedge^2\mathbb{C}^5)\cong\mathbb{P}^{10}\) be the cone over \(\text{Gr}(2,5)\) of vertex \(\nu\coloneqq\mathbb{P}(\mathbb{C})\).
A smooth Gushel-Mukai (GM for short)  variety of dimension \(3\leq n\leq 6\) is a smooth dimensionally transverse intersection
\[
X=\text{CGr}(2,5)\cap\mathbb{P}(W)\cap Q,
\]
where \(W\) is a vector space of dimension \(n+5\) of \(\mathbb{C}\oplus\wedge^2\mathbb{C}^5\) and \(Q\) is a quadric hypersurface in \(\mathbb{P}(W)\cong\mathbb{P}^{n+4}\).
Smooth GM varieties are isomorphic to either quadric sections of a linear section of \(\text{Gr}(2,5)\subseteq\mathbb{P}^9\) (called ordinary GM varieties) when \(3\leq n\leq 5\), or to double covers of a linear section of \(\text{Gr}(2,5)\) branched along a quadric section (called special GM varieties). 
We note that such \(X\) is a Fano manifold of Picard number 1 and of Fano index \(n-2\).
By our assumption, \(n\neq 6\). 
So we may pick an ordinary GM manifold \(X\);  then its VMRT is the intersection of the VMRT of the linear section of \(\text{Gr}(2,5)\) and the VMRT of the quadric hypersurface, which is smooth and not dual defective (cf.~\cite[Proposition 3.1]{Ein86} and \cite[Theorem 5.3]{Tev05}). 
In particular, it follows from \cite[Lemma 4.14]{Liu23} that the VMRT of a special GM manifold is not dual defective, either.
By Theorem \ref{thm-main} and \cite[Proposition 4.16]{Liu23}, the tangent bundle \(T_X\) is not big.
\end{proof}

\begin{proof}[Proof of Proposition \ref{prop-Mukai-nonbig}]
By \cite[Corollary 1.2 or Corollary 1.3]{HL23}, we may assume that the dimension of $X$ is at least 4.
Recall that \(g(X)\geq 2\).
We refer the reader to \cite{Muk89} for a full classification of Mukai manifolds. 
If $g(X)=6$, then $X$ is a Gushel-Mukai manifold and the result follows from Proposition \ref{prop-GM}. 
    If $g(X)=5$ (resp.\,$g(X)=4$), then $X$ is a smooth complete intersection of three quadrics (resp. smooth complete intersection of a cubic and a quadric), which has a non-big tangent bundle by Theorem \ref{thm-complete-intersection}.

\par \vskip 0.5pc \noindent
\textbf{The case $g(X)=3$.} 
Then $X$ is either a smooth quartic hypersurface in a projective space or a double cover of a quadric $Q^n$ branch along the intersection of $Q^n$ with a quartic hypersurface. 
In the former case,  \(T_X\) is not big by  \cite[Theorem 1.3]{HLS22}.
    So we may assume that \(X\) is a double cover over a quadric of dimension \(\ge4\). 
     By \cite[Theorem 1.1 and Proposition 4.3]{Kim16}, the VMRT along a general point of \(X\) is smooth and there is a special one whose VMRT along a general point is a smooth complete intersection which is not dual defective. 
     By Theorem \ref{thm-main}, the tangent bundle of this special one is not big; hence our result follows from \cite[Proposition 4.16]{Liu23}.

\par \vskip 0.5pc \noindent
\textbf{The case $g(X)=2$.} 
Then $X$ is a smooth hypersurface of degree 6 in the weighted projective space $\mathbb{P}(3,1,...,1)$. 
Let \([y:x_0:\cdots:x_n]\) be the coordinates of the weighted projective space.
After a suitable coordinate base change, we may assume that \(X\) is defined by the equation \(y^2+F_6(x_0,\cdots,x_n)=0\) where \(F_6\) is a degree 6 homogeneous polynomial.
Let $X'$ be the hypersurface of degree 6 in the projective space $\mathbb{P}^{n+1}$ defined by $Y^6+F_6(y_0,\cdots,y_n)=0$, where \([Y:y_0:\cdots:y_n]\) are the coordinates of \(\mathbb{P}^{n+1}\). 
Then there is a finite morphism $g\colon X'\to X$ sending the point $[Y:y_0:\cdots:y_n]$ to the point $[Y^3:y_0:\cdots:y_n]$. 
Clearly, \(X'\) is also smooth. 
Consider the natural injection 
$$0\to g^*\Omega_X\to \Omega_{X'}.$$
Taking the $(n-1)$-th exterior power of the above exact sequence and then tensoring  with $(n-2)H_X$, we get an injection:
\[0\to g^*T_X\to T_{X'}(2).\]
By \cite[Theorem 1.3]{HLS22}, for any \(k\in\mathbb{N}\), we have \(H^0(X',\Sym^k(T_{X'}(3)))=0\).
Therefore, \(T_{X'}(2)\) is not pseudo-effective, for otherwise, \(T_{X'}(3))\) would be big, a contradiction.  
Applying \cite[Lemma 2.2]{HLS22} to the injection \(0\to g^*T_X\to T_{X'}(2)\), we see that \(T_X\) is not big (and even not pseudo-effective), either. 
\end{proof}

\appendix

\section{Finite covering of intersections of two quadrics}

\begin{center}
{\small \sc By Jie Liu}
\end{center}

\renewcommand{\thethm}{\Alph{section}.\arabic{thm}}

The aim of this appendix is to prove the following result:

\begin{theorem}\label{thm-intersection-quadrics}
	Let $f\colon Y\rightarrow X$ be a general smooth finite cover of a general complete intersection $X$ of two quadrics in $\mathbb{P}^{n+2}$ with $n\geq 2$. Then $T_Y$ is not pseudo-effective.
\end{theorem}
Here by a \emph{general finite cover} we mean that there exists an irreducible component $B$ of the branch locus of $f$ such that $B$ is a general smooth element in the complete linear system $|\mathcal{O}_X(d)|$ for some $d\geq 1$, where $\mathcal{O}_X(1)\cong \mathcal{O}_{\mathbb{P}^{n+2}}(1)|_X$.

\subsection{Intersection of two quadrics}

Let $X$ be a smooth complete intersection of two quadrics in $\mathbb{P}^{n+2}$ with $n\geq 2$. In an appropriate system of coordinates $(x_0,\dots, x_{n+2})$, the variety $X$ is defined by the two equations $q_1=q_2=0$, where 
\begin{center}
	$q_1=\sum x_i^2$, $q_2=\sum \mu_i x_i^2$ with $\mu_i$ distinct.
\end{center}
Let $\zeta_X$ be the tautological divisor of $\mathbb{P}(T_X)$. By \cite[Theorem, a)]{BEHLV23}, the complete linear system $|2\zeta_X|$ is defined by a dominant rational map
\[
\varphi\colon\mathbb{P}(T_X)\dashrightarrow \mathbb{P}^{n-1}
\]
such that \(2\zeta_X\sim\varphi^*\mathcal{O}(1)\), noting that \(2\zeta_X\) is movable (see \cite[Corollary 3.3]{BEHLV23}). 
Denote by $Z$ the base locus of $|2\zeta_X|$. Let $B\in |\mathcal{O}_X(d)|$ be a general smooth hypersurface of degree $d\geq 1$ and let $s\colon B\rightarrow \mathbb{P}(T_X)$ be the section corresponding to the natural quotient $T_X|_B\rightarrow N_{B/X}$.

\begin{lemma}
	\label{l.dominant}
	$\varphi(s(B))=\mathbb{P}^{n-1}$.
\end{lemma}

\begin{proof}
	Firstly we assume that $n=2$. Then $X$ is a del Pezzo surface which contains exactly $16$ lines, namely $l_i$ $(1\leq i\leq 16)$. Denote by $\bar{l}_i\subset \mathbb{P}(T_X)$ the section of $\mathbb{P}(T_X)$ over $l_i$ corresponding to the quotient $T_X|_{l_i}\rightarrow N_{l_i/X}$. Then we have $Z=\cup \bar{l}_i$ by \cite[\S,3.2]{HLS22}. On the other hand, as $B$ is general, we may assume that $B$ is not tangent to $l_i$ to any points; that is, $s(B)$ is disjoint from $\bar{l}_i$ and thus $s(B)$ is disjoint from $Z$. In particular, the map $\varphi$ is well-defined along $s(B)$. On the other hand, since $\zeta_X|_{B}\cong \mathcal{O}_{B}(d)$ is ample, we obtain that $\varphi(s(B))=\mathbb{P}^1$.
	
	Now we assume that $n\geq 3$. Note that it suffices to prove that $s(B)$ is not contained in any element in $|2\zeta_X|$. Let $X'$ be the complete intersection $X\cap \{x_0=0\}$. Then $X'$ is a smooth intersection of two quadrics in $\mathbb{P}^{n+1}$. Moreover, the involution is defined as
	\[
	(x_0,x_1,\cdots,x_{n+2}) \mapsto (-x_0,x_1,\cdots,x_{n+2})
	\]
	induces a canonical splitting $T_X|_{X'}\cong T_{X'}\oplus N_{X'/X}$ (see \cite[\S,7.3]{BEHLV23}). Then it follows from \cite[Proposition 7.6]{BEHLV23} that the induced map
	\[
	H^0(X,\Sym^2 T_X) \rightarrow H^0(X',\Sym^2 T_{X'})
	\]
	is surjective and its kernel is generated by a non-zero element $\sigma\in H^0(X,\Sym^2 T_X)$. By Kodaira's vanishing theorem, the restriction 
	\[
	H^0(X,\mathcal{O}_{X}(d))\rightarrow H^0(X',\mathcal{O}_{X'}(d))
	\]
	is surjective. In particular, we may assume that $B'\coloneqq X'\cap B$ is also a general element in $|\mathcal{O}_{X'}(d)|$. Let $s':B'\rightarrow \mathbb{P}(T_{X'})$ be the section corresponding to $T_{X'}|_{B'}\rightarrow N_{B'/X'}$. Then we have $s(B)\cap \mathbb{P}(T_{X'})=s'(B')$. By induction on $n$ and using the natural isomorphism 
	\[
	H^0(X',\Sym^2 T_{X'}) \cong H^0(\mathbb{P}(T_X'),\mathcal{O}_{\mathbb{P}(T_X')}(2\zeta_{X'})),
	\]
	we only need to show that $s(B)$ is not contained in the divisor $\mathscr{C}\in |2\zeta_X|$ on $\mathbb{P}(T_X)$ defined by $\sigma$, which follows immediately as $B$ is general.
\end{proof}

\subsection{Proof of Theorem \ref{thm-intersection-quadrics}} 

Denote by $R$ the ramification divisor of $f$. 
Consider the following exact sequence of sheaves
\[
0\rightarrow T_Y\rightarrow f^*T_X \rightarrow \mathscr{Q}\rightarrow 0,
\]
where $\mathscr{Q}$ is supported on $\textup{Supp}(R)$. 
Moreover, there exists a closed subset $W$ of $X$ with codimension two such that 
\[
\mathscr{Q}|_{R\setminus \overline{W}} \cong \mathcal{O}_{R}(R)|_{R\setminus \overline{W}},
\]
where $\overline{W}=f^{-1}(W)$. 
Let $\mu\colon M\rightarrow \mathbb{P}(T_X)$ be the resolution of the indeterminacy locus of $\varphi$. 
Then we get the following commutative diagram
\[
\begin{tikzcd}[row sep=large,column sep=large]
	M \arrow[d,"\mu"] \arrow[dr,"\nu"] &     \\
	\mathbb{P}(T_X) \arrow[r,dashed, "\varphi"] \arrow[d,"p_X"] & \mathbb{P}^{n-1}\\
	X  & 
\end{tikzcd}
\]
Let $W'$ be the strict transformation of $p_X^{-1}(W)$ in $M$. 
Then $W'$ has codimension two in \(M\).  
In particular, the intersection $F'\cap W$ has also codimension two in $F'$, where $F'$ is a general fibre of $\nu$. 
It follows that $F\cap p_X^{-1}(W)$ has codimension two in $F$, where $F=\mu(F')$. 
As $X$ is general, by \cite[Theorem, c)]{BEHLV23}, the general fiber of the Lagrangian fibration $\Phi: T_X^*\to \mathbb{C}^n$ is of the form $A\setminus Z$, where $A$ is an abelian variety and $codim(Z)\geq 2$. Note that there is an \'etale double cover from the general fiber of the Lagrangian fibration $\Phi: T_X^*\to \mathbb{C}^n$ to $F\setminus Z$ (see the explanation between  Proposition 5.5 and Lemma 5.6 in \cite{BEHLV23}). Hence  $F\setminus Z$ is an open subset of a quasi-\'etale quotient of an abelian variety by the involution,  whose complement is of codimension at least two. 
As a consequence,  $F_{\circ}\coloneqq F\setminus (Z\cup p_X^{-1}(W))$ is also an open subset of a quasi-\'etale quotient of an abelian variety by the involution, whose complement is of codimension at least two. 
Let $B$ be an irreducible component of the branch locus of $f$ which is a general hypersurface of degree $d\geq 1$. In particular, we may assume that $s(B)$ is not contained in $Z\cup p_X^{-1}(W)$ (see \cite[Proposition 3.1]{BEHLV23}). Then it follows from Lemma \ref{l.dominant} that the intersection $s(B)\cap F_{\circ}$ is non-empty. As a consequence, there exists a dominating family $\{C_t\}_{t\in T}$ of irreducible curves on $\mathbb{P}(T_X)$ such that a general member $C_t$ is contained in $F_{\circ}$ for some general $F$ and meets $s(B)$. 
In particular, we have $\zeta_X\cdot C_t=0$ as \(\varphi\) is the Iitaka fibration.

Denote by $\bar{f}\colon\mathbb{P}(f^*T_X)\rightarrow \mathbb{P}(T_X)$ the induced morphism and let $\overline{\zeta}_X$ be the tautological divisor of $\mathbb{P}(f^*T_X)$. Then we have $\bar{f}^*\zeta_X=\overline{\zeta}_X$. Denote by $Y_{\circ}$ the open subset $Y\setminus \overline{W}$ whose complement is also of codimension two in \(Y\). Then the induced exact sequence
\[
0\rightarrow T_Y|_{Y_\circ} \rightarrow f^*T_X|_{Y_{\circ}} \rightarrow \mathscr{Q}|_{R\cap Y_{\circ}} \cong \mathcal{O}_R(R)|_{Y_\circ} \rightarrow 0
\]
induces an elementary transformation of vector bundles (see \cite[\S\,1]{Mar73}) and it induces a commutative diagram
\[
\begin{tikzcd}[row sep=large,column sep=large]
	 &  \Gamma \arrow[dl,"g" above] \arrow[dr,"h"] &   &  \\
	 \mathbb{P}(T_Y|_{Y_{\circ}})   &  & \mathbb{P}(f^*T_X|_{Y_{\circ}}) \arrow[ll,dashed]\arrow[r,"\bar{f}_{\circ}"] & \mathbb{P}(T_X|_{X\setminus W}),
\end{tikzcd}
\]
where $h$ is the blowing-up of the subscheme $D\coloneqq \mathbb{P}(\mathcal{O}_R(R)|_{Y_{\circ}})\subset \mathbb{P}(f^*T_X|_{Y_{\circ}})$. Let $E$ be the exceptional divisor of $h$. By \cite[Theorem 1.1]{Mar73}, we have
\[
\mathcal{O}_{\Gamma}(g^*\zeta_Y) \cong \mathcal{O}_{\Gamma}(h^*\overline{\zeta}_X-E).
\]

As $\bar{f}_{\circ}$ is surjective and $h$ is birational, the family $\{C_t\}_{t\in T}$ of curves on $\mathbb{P}(T_X|_{X\setminus W})$ can be lifted to a dominating family $\{C_{\tilde{t}}\}_{\tilde{t}\in \tilde{T}}$ of irreducible curves on $\Gamma$. Moreover, since $C_t$ always meets $s(B)$, by the definition of $D$ and $s(B)$, there exists an irreducible component $E'$ of $E$ such that $s(B)$ is the closure of $\bar{f}_{\circ}(h(E'))$ and $C_{\tilde{t}}$ always meets $E'$. In particular, we have $E\cdot C_{\tilde{t}}>0$ (see \cite[Theorem 2.2]{BDPP}). 
As $\zeta_X\cdot C_t=0$, we get $h^*\overline{\zeta}_X\cdot C_{\tilde{t}}=0$ by the projection formula. In particular, we obtain
\[
g^*\zeta_Y\cdot C_{\tilde{t}} =(h^*\overline{\zeta}_X-E)\cdot C_{\tilde{t}} < 0.
\]
This implies that $g^*\zeta_Y$ is not pseudo-effective and thus $\zeta_Y$ is not pseudo-effective, either.

\bibliography{ref}

\end{document}